\documentclass[envcountsect,envcountsame,11pt]{llncs}

\usepackage{amsmath,amssymb,amsfonts}
\usepackage{beramono,eulervm}
\usepackage{varioref}
\usepackage{enumerate}
\usepackage{empheq}
\usepackage{geometry}
\newcommand{\bqed}{\hfill{$\blacksquare$}}

\newcommand{\es}{\varnothing}

\DeclareMathOperator{\AT}{\mathsf{AT}}

\DeclareMathOperator{\I}{\mathrm{\,I\,}}
\DeclareMathOperator{\numbad}{\mathrm{numBad}}

\begin{document}

\title{{\sc Gray Codes for $\AT$-Free Orders}\thanks{A preliminary version ``Gray codes for $\AT$-free orders via antimatroids'' was presented in the 26th International Workshop on Combinatorial Algorithms (IWOCA 2015), Verona, Italy.}}

\author{
Jou-Ming~Chang\inst{1} 
\and 
Ton~Kloks\inst{}
\and 
Hung-Lung~Wang\inst{1}
}

\institute{
Institute of Information and Decision Sciences\\
National Taipei University of Business, Taipei, Taiwan\\
{\tt (spade,hlwang)@ntub.edu.tw}
}

\maketitle 

\begin{abstract}
$\AT$-free graphs are characterized by vertex elimination orders. 
We show that these $\AT$-free orders of a graph  
can be generated in constant amortized time. 
\end{abstract}

\section{Introduction}

$\AT$-free graphs are those that do not have an asteroidal triple 
\,\textemdash\, that is \,\textemdash\, 
$\AT$-free graphs do \,not\, have 
three vertices \:of \,which \,every pair is connected by a path \,that 
\,avoids \,the \,neighborhood \,of \,the \,third. 

\bigskip 

Broersma \,et al.~\cite{kn:broersma} 
\,introduced the following `betweenness relation' 
to characterize $\AT$-free graphs. Let $G$ be an $\AT$-free graph. 
For a vertex $x$ \:let $N(x)$ denote its neighborhood and let 
$N[x]$ denote its closed neighborhood; \,$N(x) \cup x$. 
A vertex $z$ is \,\underline{between}\, $x$ and $y$ if there is a path from 
\,$z$\, to \,$x$\, \,that avoids \,$N[y]$ \,and \,similarly \:
there is a path from \,$z$\, to \,$y$\, \,that avoids \,$N[x]$.  
\par 
Let \,$\I(x,y)$\, stand for the set of vertices that are between $x$ and $y$. 
\,Then \,a graph is $\AT$-free \,if and only if 
\,for any three vertices $x$, $y$ and $z$ the following property holds.  
\[z \,\in\, \I(x,y) \quad \Rightarrow \quad x \,\notin\, \I(z,y)\] 

\bigskip 

A \,set system \,is a pair $(E,\mathcal{C})$ \;where $E$ is a finite set and 
$\mathcal{C}$ is a collection of subsets of $E$. 
The elements of $\mathcal{C}$ will be called \,\underline{convex}. 
The problem to determine 
\;for which set systems a greedy algorithm optimizes 
linear objective functions   
\;has a long history \,\textemdash\, for a brief overview see \,eg\, Helman 
\,et al.~\cite{kn:helman} \;For set systems that are convex geometries 
\:Kashiwabara \,and \,Okamoto~\cite{kn:kashiwabara} \:characterize 
\,linear programming problems  
\,for which a greedy algorithm
finds an optimum. 
 
\bigskip 

A \,\underline{convex geometry}\, is a set system $(E,\mathcal{C})$ that 
satisfies the following properties. 
\begin{enumerate}[\rm (1.)]
\item $E \in \mathcal{C}$ and $\es \in \mathcal{C}$. 
\item $\mathcal{C}$ is closed under intersections. 
\item
The \underline{anti-exchange property} holds, that is, 
for all convex sets $Y \in \mathcal{C}$ \;and \;$x,z \notin Y$, \;$x \neq z$  
\begin{gather*}
z \,\in\, \sigma(x+Y) \quad \Rightarrow\quad x \,\notin\, \sigma(z+Y) \\
\text{where we write}
\quad \sigma(U)=\;\bigcap \;\;\{\:J\:|\: J \in \mathcal{C} 
\quad \text{and}\quad 
U \subseteq J\;\} \quad\text{for a subset $U$ of $E$.}
\end{gather*}
\end{enumerate}

\newpage
\noindent For \:some \:interesting \:`prospective applications' 
\:of convex geometries \:in \:cloud computing 
\;we \:refer \:to \:Kordecki~\cite{kn:kordecki}. 

\bigskip 

Let $G$ be an $\AT$-free graph. Define a set system \:on $V=V(G)$ 
\:as the collection of convex sets in $G$ \,\textemdash\, 
where a set $X \subseteq V$ \:is \:convex \,if 
it contains \:with any two of its elements \:the elements  
that are between them. 
\par 
In the following section we show that 
\,the collection of convex sets 
\,in an $\AT$-free graph \,constitutes  \,a \,convex \,geometry.  
This completes the 
result of Alc\'on \,et al.~\cite{kn:alcon}\, who proved a similar result 
for interval graphs. 

\section{$\AT$-free convex geometries}

A set system $(E,\mathcal{C})$ \:which satisfies $E,\es \in \mathcal{C}$ 
\:and \:which is closed under intersections \:is called  
an alignment \:by \:Edelman and Jamison~\cite{kn:edelman}.  They show that an alignment 
satisfies the anti-exchange property \:if and only if \:either one 
\:of the following two properties holds. 
\begin{enumerate}[\rm (1.)] 
\item For any $C \in \mathcal{C}$ and $y \notin C$ \:the element  
$y$ is \,\underline{extreme}\, in $\sigma(C + y)$ 
\,\textemdash\, that is \,\textemdash\, 
$\sigma(C + y) \setminus y$ is convex. 
\item Any convex set $X$,  $X\neq E$, has an element $y \notin X$ 
such that $X + y$ is convex. 
\end{enumerate}

\bigskip 

The following definition allows us to characterize convex geometries 
with a third property 
(which is equivalent to the 
characterization~\cite[Theorem~2.3]{kn:edelman}). 

\begin{definition}
Let \,$(E,\mathcal{C})$\, be an alignment, \;let \,$X,Y \subseteq E$ \,and \,let 
\:$Y=\{y_1,\dots,y_k\}$. \,Let \,$k \geq 2$. 
\;The set \,$Y$ \,\underline{induces a cycle}\, on \,$X$\, if 
\:for \:all \:$i \in [k]$    
\[y_{i+1} \in \sigma(y_i+X) \quad \text{\rm{if \,$i < k$ \;and}}\quad 
y_1 \in \sigma(y_k+X).\]
\end{definition}

\bigskip 

\begin{lemma} 
\label{property (2.3)} 
An alignment \,$(E,\mathcal{C})$\, is a convex geometry \,if and only if 
\,any set \,$Y$\, \:which induces a cycle on a set \,$X$ 
\:is contained in \,$\sigma(X)$. 
\end{lemma}

\bigskip 

All proofs we skip here can be found in the appendix. We proceed to prove \,that the convex sets in an $\AT$-free graph constitute 
a convex geometry 
\,via the following three lemmas. 
(Some easily-made drawings might be helpful to the reader.)

\begin{lemma}
\label{prop 3.3}
Let $G$ be an $\AT$-free graph. Any four vertices satisfy the following 
property. 
\begin{equation}
u \in \I(v,x) \quad \text{\rm{and}}\quad 
v \in \I(u,y) \quad \Rightarrow \quad u \in \I(x,y)
\end{equation}
\end{lemma}

\bigskip 

\begin{lemma}
\label{lm 3.4}
Let \,$G$\, be an \,$\AT$-free graph. 
Any five  vertices satisfy the following 
property. 
\begin{multline}
a \in \I(x,b) \quad \text{\rm{and}} \quad b \in \I(y,z) 
\quad \text{\rm{and}} \quad a \notin N[y] \cup N[z] \quad \Rightarrow \\
a \in \I(x,y) \quad\text{\rm{or}}\quad a \in \I(x,z) \quad\text{\rm{or}}\quad 
a \in \I(y,z)
\end{multline}
\end{lemma}

\bigskip 

(In~\cite{kn:chvatal} \:Chvat\'al\,  describes a subclass of convex geometries by a property similar to Lemma~\ref{lm 3.4}.) A component of a graph is a maximal subset of vertices \;of \:which  
\;every pair is connected by a path. 
\begin{lemma}
\label{lemma 3.5}
Let \,$G$\, be an \,$\AT$-free graph. Let \,$X,Y \subseteq V$ \,and let  
\:$Y=\{y_1,\dots,y_k \}$. 
Assume that \,$Y \cap \sigma(X)= \es$\, and that no subset of \,$Y$\, 
induces a cycle on \,$X$. \:If \,\textemdash\, for all \,$i \in [k-1]$ 
\,\textemdash\, there exist $x_i \in X$ such that 
\:$y_{i+1} \in \I(y_i,x_i)$ \:then  
\[Y \,\subseteq\, N[y_1] \,\cup\, C
\quad \text{\rm{where $C$ is a component of the graph $G-N[y_1]$.}}\] 
\end{lemma}

\bigskip 

\begin{theorem}
Let \,$G$\, be an \,$\AT$-free graph. 
\,The convex sets in \,$G$\, constitute a 
convex geometry \,on \,$V(G)$. 
\end{theorem}

\section{Generating $\AT$-free orders}

When $(E,\mathcal{C})$ is a convex geometry then 
$(E,\Bar{\mathcal{C}})$ is an \,\underline{antimatroid}\, 
and this defines all 
antimatroids \,\textemdash\, here we write 
\[\Bar{\mathcal{C}}=\{\,E \setminus C\;|\; C \in \mathcal{C}\,\}.\] 
Crapo~\cite{kn:crapo} \,characterizes formal languages    
that are antimatroids \,as \,follows. 

\begin{definition}
A language \,$L$\, is an antimatroid 
\,if \,its \,words \:satisfy \:the following properties. 
\begin{enumerate}[\rm (1.)]
\item Every symbol of the alphabet occurs in at least one word. 
\item Every word of \,$L$\, contains at most one copy of 
every symbol in the alphabet. 
\item Every prefix of a word in \,$L$ \,is \,in \,$L$. 
\item If \,$s,t \in L$ \,and \,if \,$s$ contains at least one 
symbol that is not in \,$t$ 
\,then \,there is a symbol \,$x \in s$ such that \,$tx \in L$. 
\end{enumerate} 
\end{definition}

\bigskip 

\,\textemdash\, Observe that \,\textemdash\, when $L$ is 
the language \:whose words are prefixes of  $\AT$-free orders of a graph 
\:then \:$L$ \:is \:an \:antimatroid. 
The \underline{basic words} of $L$ are those of maximal 
length \;which \:are \:the \:$\AT$-free orders. 

\begin{definition}
A linear order \:$<$ \:of the vertices 
of a graph \:is an 
\:$\AT$-free order \:if any three vertices satisfy the following property. 
\[z \in \I(x,y) \quad \Rightarrow\quad 
x < z  \quad \text{\rm{or}}\quad 
y < z\] 
\end{definition}
A graph is $\AT$-free \,if and only if\, it has 
an $\AT$-free order~\cite{kn:corneil}.

\bigskip 

Pruesse \,and \,Ruskey~\cite{kn:pruesse2,kn:pruesse1} considered the problem of 
producing a Gray code for the 
basic words of an antimatroid.  
\par
Let $L$ be an antimatroid. 
Consider the graph whose vertices are the basic words of $L$ 
\;two vertices being adjacent when one is obtained 
from the other by a 
transposition of an adjacent pair.   
The \,\underline{prism}\, is obtained from two copies 
($+$ and $-$) of this graph \:and the addition of edges joining 
$+/-$ copies of similar vertices. 
\:Pruesse \:and \:Ruskey \:show that this prism is 
Hamiltonian \:for \:all \:antimatroids. 
\:Their generic algorithm \:generates 
\:all the basic words of $L$ \:in the order of a Hamiltonian traversal  
of the prism \,\textemdash\, whilst reporting only 
the (transpositions in the) $+$ copies.   

\bigskip 

Assume \,that a graph \,$G$\, is connected and $\AT$-free. 
Let $\omega$ be a vertex \:such \:that  
the number of vertices \:in the largest component of $G-N[\omega]$ 
\:is \:as \:large \:as \:possible. 
\:Let \:$C$ be a largest component of $G-N[\omega]$ 
\:and \:let \:$S=N(C)$ and $\Omega=V\setminus (C \cup S)$. 
\,When $G$ is not a clique then $\{\Omega,S,C\}$ is a 
partition of $V(G)$. 
\par 
By our choice of $\omega$ \:every vertex of $\Omega$ is 
adjacent to every vertex of $S$ \,\textemdash\, that is 
\,\textemdash\, $\Omega$ is a module, \:hence, \:convex. 
Since $G$ is $\AT$-free \:$C \cup S$ is convex as well. 

\bigskip 

Consider $\AT$-free orders for $G[\Omega]$ and 
$G[C \cup S]$ \,\textemdash\, say 
\[a=a_1\,\dots \,a_n \quad\text{and}\quad 
b=b_1\,\dots\,b_m.\]
Notice that the linear order 
\[\beta=a_1\,\dots\,a_n\,b_1\,\dots\,b_m\] 
is an $\AT$-free order for $G$ \,\textemdash\, we 
call \:$\beta$ \:a \:\underline{canonical order} 
\:when $a$ and $b$ are that of $G[\Omega]$  and $G[C \cup S]$. 
\;It is easy to obtain 
a canonical order in polynomial time. 

\bigskip 

Consider an arbitrary order $\sigma=v_1\,\dots\,v_n$ of the 
vertices of $G$. We write $\sigma_1$ for the linear order 
induced on $\Omega$  
and $\sigma_2$ for the linear order induced on $C \cup S$. 
Observe that 
$\sigma$ is an $\AT$-free order \:if \:and \:only \:if 
\begin{enumerate}[\rm 1.] 
\item $\sigma_1$ and $\sigma_2$ are 
$\AT$-free orders of $\Omega$ and $C \cup S$ 
\item for $\omega \in \Omega$ and $x,y \in C\cup S$ 
\[x \in \I(\omega,y) \quad \text{and}\quad 
\sigma_2^{-1}(x) < \sigma_2^{-1}(y) \quad \Rightarrow \quad 
\sigma^{-1}(\omega) < \sigma^{-1}(x)\] 
\end{enumerate}
that is \,\textemdash\, \:all \:vertices of $\Omega$ should appear 
\:before \:the first element of a pair $x,y \in C \cup S$ 
that satisfies $x \in \I(\omega,y)$ \:and \:that \:is 
\:in \:a \:`wrong' \:order 
\,\textemdash\, namely \,\textemdash\, 
$\sigma^{-1}(y) > \sigma^{-1}(x)$. 

\bigskip 

In \:proving \:that \:the prism of $G$ is Hamiltonian \:we 
may assume that the prisms of 
\,$\Omega$\, and \,$C \cup S$ \:are \:that. 
\,\textemdash\, Furthermore \,\textemdash\, we may 
assume that $\{+\beta,-\beta\}$ 
is an edge of 
\:both  \:Hamiltonian \:cycles. 
\:A Hamiltonian cycle in the prism of $G$ 
\:that uses the edge $\{+\beta,-\beta\}$ \:is 
easily obtained from this~\cite[Theorem~3.3]{kn:pruesse1}. 
\par 
It 
\:follows \:that 
\:the 
$\AT$-free orders of \,$G$\, can be 
generated \:such \:that \:each order differs from its predecessor 
\:by \:at \:most \:one \:or \:two \:adjacent \:transpositions. 
\par
It remains to establish the timebound.  

\bigskip 

\begin{theorem}
The \:$\AT$-free orders \:of an \:$\AT$-free graph can be generated \:in 
\:con\-stant \:amortized \:time.
\end{theorem}
\begin{proof}
Pruesse and Ruskey developed a generic algorithm 
to produce all basic words of an antimatroid~\cite{kn:pruesse2,kn:pruesse1}. 
The amortized time complexity is determined by 
an \,\textemdash\, antimatroid specific 
\,\textemdash\, transposition oracle\: 
which answers whether two adjacent elements 
in a basic word may \:swap \:places \:to produce 
another basic word. 

\medskip 

\noindent 
We use the notation introduced above. 
For an $\AT$-free order $\sigma$ \:with an induced 
order $\sigma_2$ on $C\cup S$ \:and $x \in C$ 
\:define \,$h(x)$\, \:as \:follows.  
\[h(x) = \# \;\{\,z\;|\; z \in C \quad \text{and}\quad 
x \in \I(\omega,z) \quad \text{and}\quad 
\sigma_2^{-1}(x) < \sigma_2^{-1}(z)\;\}\]
Then $\sigma$ is an $\AT$-free order \,if \,and \,only \,if 
\,$\sigma_1$ and $\sigma_2$ \:are \:that \:and 
\:$\sigma^{-1}(\omega) < \sigma^{-1}(x)$ when $h(x) \geq 1$. 
We show that $h$ can be maintained 
\:during a swap of two adjacent elements in $\sigma$. 
\:Notice \:that $h$ is easily computable for a canonical order. 

\medskip 

\noindent 
Sawada~\cite[Theorem~15~ff.]{kn:sawada} 
\:introduces  the counter $\numbad(x,y)$ \:for ordered pairs 
$x$ and $y$ \:as the number of vertices $z$ with $x \in \I(y,z)$ 
and $\sigma^{-1}(z) > \sigma^{-1}(x)$. Two elements 
\,$v_j$ and $v_{j+1}$\, can be swapped to produce a new 
$\AT$-free order \:only \:if 
\:$\numbad(v_{j+1},v_j)=0$. \:Sawada \:shows that 
$\numbad$ can be maintained 
during a generation of $\AT$-free orders in constant 
amortized time~\cite[Theorem~13 and Observation~1]{kn:sawada}. 

\medskip 

\noindent 
Notice that $h(x)=\numbad(x,\omega)$. This proves the theorem. 
\bqed\end{proof}

\section{Concluding remark}

The family of ideals \,in a poset \, constitutes \,a convex geometry on the elements 
of the poset.  This convex geometry is usually referred to as a poset shelling. 
A convex geometry is a poset shelling \,if and only if \,its family of 
convex sets is closed 
under unions~\cite{kn:korte}.  
(See~\cite{kn:kashiwabara} for other characterizations.)

\bigskip 

The family of convex sets of the $\AT$-free graph \,shown in the figure 
\,is not a poset shelling. \:To see that \:let $A=\{y_1,z_2,u\}$ and 
let $B=\{y_2,z_1,u\}$. Then $A$ and $B$ are convex \,but \,their union 
is not \,since $u^{\prime} \in \I(z_1,z_2)$.   

\begin{figure}
\begin{center}
\setlength{\unitlength}{1mm}
\thicklines
\begin{picture}(50,50)
\put(10,30){\circle*{1}}
\put(20,40){\circle*{1}}
\put(30,10){\circle*{1}}
\put(30,20){\circle*{1}}
\put(30,30){\circle*{1}}
\put(40,40){\circle*{1}}
\put(50,30){\circle*{1}}
\put(30,10){\line(0,1){20}}
\put(30,30){\line(-1,1){10}}
\put(30,30){\line(1,1){10}}
\put(20,40){\line(1,0){20}}
\put(30,20){\line(-2,1){20}}
\put(30,20){\line(2,1){20}}
\put(10,30){\line(1,1){10}}
\put(50,30){\line(-1,1){10}}
\put(18,42){$y_2$}
\put(40,42){$y_1$}
\put(8,28){$z_2$}
\put(50.5,28){$z_1$}
\put(29,7.5){$u$}
\put(31,29){$u^{\prime}$}
\end{picture}
\end{center}
\caption{The family of convex 
sets in this graph is not a poset shelling.} 
\end{figure}
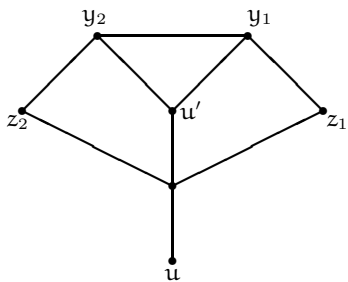


\newpage 

\section*{Appendix: proofs}

\setcounter{section}{2}
\setcounter{lemma}{1} 

\begin{lemma} 
An alignment \,$(E,\mathcal{C})$\, is a convex geometry \,if and only if 
\,any set \,$Y$ \:which induces a cycle on a set 
\,$X$ \:is contained in \,$\sigma(X)$. 
\end{lemma}
\begin{proof}
When $Y$ is a cycle on $X$ then $\sigma(X+y) \setminus y$ is not convex 
for any $y \in Y$. When \,$Y\setminus \sigma(X) \neq \es$\, there exists a 
vertex $y$ such that $\sigma(X+y) \setminus y$ is not convex. 
By Edelman and Jamison's characterization \:$(E,\mathcal{C})$ \,is \:not \:a 
\,convex \,geometry. 

\medskip 

\noindent 
Assume that any set \:that induces a cycle on $X$ \:is contained 
in $\sigma(X)$. 

\medskip 

\noindent 
Let $C$ be a convex set and assume that $C \neq E$. 
We show that there exists an element $y \in E \setminus C$ 
which satisfies 
\[C+y \,\in \,\mathcal{C}\] 
(notice that this proves the claim \,\textemdash\, by  
\,Edelman and Jamison's characterization of 
convex geometries). 

\medskip 

\noindent 
Let \:$Y \neq \es$ \:be an \,inclusion-minimal \,subset of $E \setminus C$ such that 
$C \cup Y$ is convex. We claim that \:$|Y|=1$. 
\,\textemdash\,Otherwise \,\textemdash\, $Y$ has at least 
two elements. By the assumption 
\:that $Y$ is \:set\textendash inclusion \:minimal 
\[\forall_{\,y \,\in\, Y} \quad\exists_{\,y^{\prime} \,\in\, Y} 
\quad y^{\prime} \neq y \quad\text{\rm{and}}\quad y^{\prime} \in \sigma(C+y)\] 
which implies that some nonempty subset $Y^{\prime} \subseteq Y$ 
induces a cycle on $C$. \;By the assumption \,\textemdash\, 
that any set which induces a cycle on $X$ is contained 
in $\sigma(X)$ \,\textemdash\, \:$Y^{\prime}\subseteq C$ 
\:which is a contradiction. 
\bqed\end{proof}

\bigskip 

For vertices \,$x$ and $y$\, \:that are \,not adjacent \:we write \,$C^x(y)$\, 
for the component of $G-N[x]$ that contains the vertex $y$.  
\begin{lemma}
Let $G$ be an $\AT$-free graph. Any four vertices satisfy the following 
property. 
\begin{equation*}
u \in \I(v,x) \quad \text{\rm{and}}\quad 
v \in \I(u,y) \quad \Rightarrow \quad u \in \I(x,y)
\end{equation*}
\end{lemma}
\begin{proof}
If $x=y$ then the left-hand is only satisfied when 
$\{u,v,x\}$ is an asteroidal triple. 
\:Otherwise 
\[C^u(v)=C^u(y) \quad\text{and}\quad 
C^u(v) \neq C^u(x).\]
The vertices $u$ and $y$ are connected by a path that 
avoids $N[x]$ \:since $u \in \I(v,x)$. 
\,\textemdash\, Also \,\textemdash\, the vertices 
$u$ and $x$ are connected by a path that avoids $N[y]$ 
\:since $v \in \I(u,y)$. 

\medskip 

\noindent 
This proves the lemma. 
\bqed\end{proof}

\bigskip 

\begin{lemma}
Let \,$G$\, be an \,$\AT$-free graph. 
Any five  vertices satisfy the following 
property. 
\begin{multline*}
a \in \I(x,b) \quad \text{\rm{and}} \quad b \in \I(y,z) 
\quad \text{\rm{and}} \quad a \notin N[y] \cup N[z] \quad \Rightarrow \\
a \in \I(x,y) \quad\text{\rm{or}}\quad a \in \I(x,z) \quad\text{\rm{or}}\quad 
a \in \I(y,z)
\end{multline*}
\end{lemma}
\begin{proof}
Assume that $C^b(a) \neq C^b(z)$ and $C^b(a) \neq C^b(y)$. 
\:Observe that the vertices $a$ and $x$ are connected by a path that avoids $N[y] \cup N[z]$. 
If there are no paths from $a$ to $z$ \:nor from $a$ to 
$y$ \:that avoid $N[x]$ \:then 
\[N(C^b(y)) \,\cup\, N(C^b(z)) \,\subseteq\, N[x].\] 
However, there is a path from $a$ to $b$ that avoids $N[x]$ 
\:which implies \:that there is a path from 
$a$ to $b$ that avoids $N[z]$. 
By concatenation of this path \:with a path from 
$b$ to $y$ that avoids $N[z]$ 
\:we find a path from 
$a$ to $y$ that avoids $N[z]$. 
\;Similarly, \;there is a path from $a$ to $z$ that avoids $N[y]$. 
This proves $a \in \I(y,z)$. 

\medskip 

\noindent 
Assume \:$C^b(a)= C^b(z)$. \:There is a path from $a$ 
to $x$ that avoids $N[y]$. If there is no path from 
$a$ to $y$ that avoids $N[x]$ \:we have 
\[N(C^b(y)) \subseteq N[x]\] 
since there is a path connecting $a$ and $b$ that 
avoids $N[x]$. Since $a \in \I(x,b)$ \:$N[a]$ separates 
$x$ and $b$ \:which implies 
\[N(x) \,\cap\, N[b] \,\subseteq\, N(a).\] 
There is a path from $b$ to $y$ that avoids $N[z]$ 
\:since $b \in \I(y,z)$. \:Since \:$a \notin N[z]$ 
\:there is a path from $a$ to $y$ that avoids $N[z]$. 
\:Since \:$C^b(a) = C^b(z)$ \:there is a path from 
$a$ to $z$ that avoids $N[y]$ \:which implies 
\:$a \in \I(y,z)$. 

\medskip 

\noindent 
Assume \:$C^b(a)=C^b(y)$. \:In an analogous manner it 
follows that $a \in \I(x,z)$ or $a \in \I(y,z)$. 

\medskip 

\noindent 
This proves the lemma. 
\bqed\end{proof}

Note that
$a \in \I(x,b) \quad \text{\rm{and}} \quad b \in \I(y,z) 
\quad \Rightarrow \quad a \notin N[y] \cap N[z].$ 
\bigskip 

\begin{lemma}
Let \,$G$\, be an \,$\AT$-free graph. Let \,$X,Y \subseteq V$ \,and let  
\:$Y=\{y_1,\dots,y_k \}$. 
Assume that \,$Y \cap \sigma(X)= \es$\, and that no subset of \,$Y$\, 
induces a cycle on \,$X$. \:If \,\textemdash\, for 
all \,$i \in [k-1]$\,\textemdash\, \:there exist 
\:$x_i \in X$ that satisfy \:$y_{i+1} \in \I(y_i,x_i)$ \:then  
\[Y \subseteq N[y_1] \cup C
\quad \text{\rm{for a component $C$ of $G-N[y_1]$.}}\] 
\end{lemma}
\begin{proof}
Assume that the lemma does not hold. Choose $X$ and $Y$ 
so that $Y$ intersect \,at \,least \:two \:components of 
$G-N[y_1]$ \:and \:$|Y|$ \:is \:minimal. 
Then $|Y| \geq 3$. 

\medskip 

\noindent 
Since $Y$ is minimal 
\[\forall_{\,j\,>\,i+1} \quad 
\forall_{\,x \,\in \,\sigma(X)} \quad y_j \,\notin\, \I(y_i,x). \] 

\medskip 

\noindent 
By \:Lemma~\ref{lm 3.4} 
\[y_3 \in \left(\, N(\,x_1\,) \,\cup\, N(\,y_1\,) \,\right) \,\setminus\, 
\left( \,N(\,x_1\,) \,\cap\, N(\,y_1\,) \,\right).\] 

\medskip 

\noindent 
Assume $y_3 \in N(x_1)$. Consider $G-N[y_2]$. 
\:For \:$i \geq 2$ 
\[y_i \,\in\, N[\,y_2\,] \,\cup\, C^{y_2}(x_1).\] 
Since \:$C^{y_2}(x_1) \neq C^{y_2}(y_1)$ 
\:there exists of a path from $x_1$ to $y_2$ in 
$G-N[y_1]$. \:This \:contradicts \:the \:assumption. 

\medskip 

\noindent 
Assume $y_3 \in N(y_1)$. Consider $G-N[y_3]$. \:We \:have 
\:$x_2 \in N(y_1)$ \:otherwise \:$y_1 \in \I(y_2,x_2)$. 
\,\textemdash\, Moreover \,\textemdash \,  
\[C^{y_3}(x_1) \neq C^{y_3}(x_2) 
\quad\text{and}\quad C^{y_3}(x_1) \neq C^{y_3}(y_2)\] 
since the first equality would imply that $\{x_1,y_1,y_2\}$ is an 
asteroidal triple \:and 
\:the second equality would imply $y_2 \in \I(x_1,x_2)$. 

\medskip 

\noindent 
When \:$|Y| \geq 4$, 
\[y_4 \,\in\, \left( N(\,x_2\,) \,\cup\, N(\,y_2\,) \right) \setminus N[\,y_3\,] \]
and thus
\[
\forall_{\,i \,\geq\, 4} \quad y_i \,\in\, C^{y_3}(\,x_2\,) 
\,\cup\, C^{y_3}(\,y_2\,) \,\cup\, N(\,y_3\,).\]
Notice that $y_2 \in \I(x_1,y_1)$ implies that $x_1$ 
and $y_2$ are connected by a path that avoids $N[y_1]$. 
\:This \:implies 
\[N(C^{y_3}(x_1)) \setminus N[y_1] \neq \es.\] 
Let \:$u \in N(C^{y_3}(x_1)) \setminus N[y_1]$. 
\:Then 
\[\forall_{\,i \,\geq \,2} \quad 
y_i \in \left(C^{y_3}(\,x_2\,) \cup C^{y_3}(\,y_2\,) \right) 
\setminus N[\,y_1\,] \quad \Rightarrow\quad 
y_i \in N(\,u\,)\] 
since \:otherwise \:$y_1 \in \I(x_1,y_i)$. 
\,\textemdash\, Furthermore \,\textemdash\, 
for $y_i \in N[y_3] \setminus N[y_1]$ 
which implies that $N[y_i]$ intersects any $x_1,y_2$-path that avoids 
$N[y_1]$. 
This proves the lemma. 
\bqed\end{proof}

\bigskip 

\begin{theorem}
Let \,$G$\, be an \,$\AT$-free graph. \,The convex sets in \,$G$\, constitute a 
convex geometry \,on \,$V(G)$. 
\end{theorem}
\begin{proof}
Assume there exist sets $X$ and $Y$ that contradict this 
theorem \,\textemdash\, that is \,\textemdash\, $Y$ induces a cycle on $X$. \:Let \:$Y=\{y_1,\dots,y_k\}$. \:We may assume that 
$Y$ is minimum \:and \:that \:$Y \cap \sigma(X) = \es$. 
\par\noindent 
\,\textemdash\, Also \,\textemdash\, 
\[\exists_{\,x \,\in \,X} \quad y_j \in \I(y_i,x) 
\quad\Rightarrow \quad j=i+1,\] 
where all arithmetics here are taken with modulo $k$. By \:Lemma~\ref{lm 3.4} \: 
\[\forall_i \quad y_{i+2} \,\in\, 
\left(\,N(\,x_i\,) \,\cup\, N(\,y_i\,) \,\right)
\,\setminus\, \left(\,N(\,x_i\,) \,\cap\, N(\,y_i\,)\,\right).\]

\medskip 

\noindent
Consider \:the \:following \:two \:cases. 
\begin{description}
\item[Case $1$:] \;Assume $\forall_{\,i} \:y_{i+2} \in N(y_i)$. 
\:Then 
\begin{equation}
\label{eq:p1}
N(\,x_{i-1}\,) \,\cap\, N[\,y_i\,] \,\subseteq\, N(\,x_i\,) \,\cap\, N[\,y_i\,]
\end{equation}
since \:otherwise \:$y_i \in \I(x_{i-1},x_i)$. By $y_i\notin I(x_{i+1}, y_{i+1})$, 
\begin{equation}
\label{eq:p2}
y_{i+2} \in N(y_i) \quad\Rightarrow \quad 
y_i \in N(x_{i+1})
\end{equation}
Notice that Eq.~\eqref{eq:p2} holds for all $i\in [k]$. Then, by repeatedly apply Eq.~\eqref{eq:p1} we have $y_i\in N(x_j)$ for all $i$ and $j$. This contradicts the assumption that $Y$ induces a cycle on $X$. 
\item[Case $2$:] \;Assume \:$\exists_{\,j} \:y_{j+2} \in N(x_j)$. 
Without loss of generality \:assume \:$j=1$. 
We show that \:$y_2 \in \I(x_1,y_1)$ \:and \:$y_1 \in 
\I(x_1,y_2)$.  It suffices to show that $x_1$ and $y_1$ 
are connected by a path that avoids $N[y_2]$. 

\medskip 

\noindent 
Consider 
\[y_3 \in \I(x_2,y_2) \quad 
y_4 \in \I(x_3,y_3) \quad \dots \quad y_1 \in \I(x_k,y_k).\]
Since \:$\{y_1,y_3\} \cap N[y_2] =\es$ \,\textemdash\, 
by Lemma~\ref{lemma 3.5} \,\textemdash\, 
$y_1$ and $y_3$ are in the same component of $G-N[y_2]$. 
\:Since \:$x_1 \in N[y_3]$ \:this proves that $x_1$ and 
$y_1$ are connected by a path that avoids $N[y_2]$ 
\,\textemdash\, as claimed. 
\end{description}
This proves the theorem. 
\bqed\end{proof}


\begin{thebibliography}{99}

\bibitem{kn:alcon}Alc\'on,~L., B.~Bre\v{s}ar, T.~Gologranc, 
M.~Gut\'{\i}errez, T.~\v{S}umenjak, I.~Peterin and A.~Tepeh, 
Toll convexity, 
{\em European Journal of Combinatorics\/} {\bf 46} (2015), pp.~161--175. 

\bibitem{kn:broersma}Broersma,~H., T.~Kloks, D.~Kratsch and  H.~M\"uller, 
Independent sets in asteroidal triple-free graphs, 
{\em SIAM Journal on Discrete Mathematics\/} {\bf 12} (1999), pp.~276--287. 


\bibitem{kn:chvatal}Chv\'atal,~V., 
Antimatroids, betweenness, convexity. In (Cook, Lov\'asz, Vygen eds.) 

{\em Research Trends in Combinatorial Optimization\/} (2009), 
Springer, Berlin,  
Heidelberg, pp.~57--64. 

\bibitem{kn:corneil}Corneil,~D. and J.~Stacho, 
Vertex ordering characterizations of graphs of bounded 
asteroidal number, 
{\em Journal of Graph Theory\/} {\bf 78} (2015), pp.~61--79. 

\bibitem{kn:crapo}Crapo,~H., 
\,\textemdash\, Selectors \,\textemdash\, A theory of formal languages, 
semimodular lattices,  
branching and shelling processes, 
{\em Advances in Mathematics\/} {\bf 54} 
(1984), pp.~233--277. 

\bibitem{kn:edelman}Edelman,~P. and R.~Jamison, 
The theory of convex geometries, 
{\em Geometriae Dedicata\/} {\bf 19} (1985), pp.~247--270. 


\bibitem{kn:helman}Helman,~P., B.~Moret and H.~Shapiro, 
An exact characterization of greedy structures, 
{\em SIAM Journal on Discrete Mathematics\/} 
{\bf 6} (1993), pp.~274--283. 

\bibitem{kn:kashiwabara}Kashiwabara,~K. and Y.~Okamoto, 
A greedy algorithm for convex geometries, 
{\em Discrete Applied Mathematics\/} {\bf 131} (2003), pp.~449--465. 

\bibitem{kn:kordecki}Kordecki,~W., 
Secretary problem: graphs, matroids and greedoids. 
Manuscript on ArXiv: 1801.00814, 2018. 

\bibitem{kn:korte}Korte,~B., L.~Lov\'asz and R.~Schrader, 
{\em Greedoids\/}, 
Springer-Verlag Berlin Heidelberg New York, 1991. 

\bibitem{kn:pruesse2}Pruesse,~G. and F.~Ruskey, 
Gray codes from antimatroids, 
{\em Order\/} {\bf 10} (1993), pp.~239--252. 

\bibitem{kn:pruesse1}Pruesse,~G. and F.~Ruskey, 
Generating linear extensions fast, 
{\em SIAM Journal on Computing\/} {\bf 23} (1994), pp.~373--386. 

\bibitem{kn:sawada}Sawada,~J., 
Oracles for vertex elimination orderings, 
{\em Theoretical Computer Science\/} {\bf 341} (2005), pp.~73--90. 

\end{thebibliography}
\end{document}